\documentclass[12pt,oneside,a4paper]{amsart}
\usepackage[utf8]{inputenc}
\usepackage[T1]{fontenc}
\usepackage{enumitem}
\usepackage{thmtools}
\usepackage{hyperref}

\declaretheorem[numberwithin=section]{theorem}
\declaretheorem[numberlike=theorem]{lemma}
\declaretheorem[numberlike=theorem]{corollary}
\declaretheorem[numberlike=theorem]{proposition}
\declaretheorem[style=definition,numberlike=theorem]{definition}

\declaretheorem[style=remark,numberlike=theorem]{remark}

\newcommand{\terminology}[1]{\emph{#1}}
\newcommand{\RR}{\mathbb{R}}
\newcommand{\NN}{\mathbb{N}}
\newcommand{\HH}{\mathbb{H}}
\newcommand{\Lloc}{\mathrm{L}^2_{\text{loc}}}
\newcommand{\abs}[1]{\left\lvert#1\right\rvert}
\newcommand{\norm}[1]{\left\lVert#1\right\rVert}
\newcommand{\restr}[2]{\left.#1\right|_{#2}}
\newcommand{\dilationfactor}{\lambda}

\makeatletter
\newcommand*{\mint}[1]{\mint@l{#1}{}}
\newcommand*{\mint@l}[2]{\@ifnextchar\limits{\mint@l{#1}}{\@ifnextchar\nolimits{\mint@l{#1}}{\@ifnextchar\displaylimits{\mint@l{#1}}{\mint@s{#2}{#1}}}}}
\newcommand*{\mint@s}[2]{\@ifnextchar_{\mint@sub{#1}{#2}}{\@ifnextchar^{\mint@sup{#1}{#2}}{\mint@{#1}{#2}{}{}}}}
\def\mint@sub#1#2_#3{\@ifnextchar^{\mint@sub@sup{#1}{#2}{#3}}{\mint@{#1}{#2}{#3}{}}}
\def\mint@sup#1#2^#3{\@ifnextchar_{\mint@sup@sub{#1}{#2}{#3}}{\mint@{#1}{#2}{}{#3}}}
\def\mint@sub@sup#1#2#3^#4{\mint@{#1}{#2}{#3}{#4}}
\def\mint@sup@sub#1#2#3_#4{\mint@{#1}{#2}{#4}{#3}}
\newcommand*{\mint@}[4]{\mathop{}\mkern-\thinmuskip\mathchoice{\mint@@{#1}{#2}{#3}{#4}\displaystyle\textstyle\scriptstyle}{\mint@@{#1}{#2}{#3}{#4}\textstyle\scriptstyle\scriptstyle}{\mint@@{#1}{#2}{#3}{#4}\scriptstyle\scriptscriptstyle\scriptscriptstyle}{\mint@@{#1}{#2}{#3}{#4}\scriptscriptstyle\scriptscriptstyle\scriptscriptstyle}\mkern-\thinmuskip\int#1\ifx\\#3\\\else_{#3}\fi\ifx\\#4\\\else^{#4}\fi}
\newcommand*{\mint@@}[7]{\begingroup \sbox0{$#5\int\m@th$}\sbox2{$#5\int_{}\m@th$}\dimen2=\wd0\let\mint@limits=#1\relax \ifx\mint@limits\relax \sbox4{$#5\int_{\kern1sp}^{\kern1sp}\m@th$}\ifdim\wd4>\wd2 \let\mint@limits=\nolimits \else \let\mint@limits=\limits \fi \fi \ifx\mint@limits\displaylimits \ifx#5\displaystyle \let\mint@limits=\limits \fi \fi \ifx\mint@limits\limits \sbox0{$#7#3\m@th$}\sbox2{$#7#4\m@th$}\ifdim\wd0>\dimen2 \dimen2=\wd0 \fi \ifdim\wd2>\dimen2 \dimen2=\wd2 \fi \fi \rlap{$#5\vcenter{\hbox to\dimen2{\hss $#6{#2}\m@th$\hss }}$}\endgroup }
\makeatother

\newcommand{\intavg}{\mint{-}}
\newcommand{\lt}{<}
\newcommand{\gt}{>}

\title[Infinite geodesics and isometric embeddings]{Infinite geodesics and isometric embeddings in Carnot groups of step 2}
\author{Eero Hakavuori}
\address[Hakavuori]{SISSA, Via Bonomea 265, 34136 Trieste}
\email{eero.hakavuori@sissa.it}
\date{November 19, 2019}
\begin{document}
\label{x:article:step2_infinite_geodesics}

\begin{abstract}
In the setting of step 2 sub-Finsler Carnot groups with strictly convex norms, we prove that all infinite geodesics are lines. It follows that for any other homogeneous distance, all geodesics are lines exactly when the induced norm on the horizontal space is strictly convex. As a further consequence, we show that all isometric embeddings between such homogeneous groups are affine. The core of the proof is an asymptotic study of the extremals given by the Pontryagin Maximum Principle.
\end{abstract}

\thanks{The author has been partially supported by the Vilho, Yrjö and Kalle Väisälä Foundation, by the Academy of Finland (grant 288501 `\emph{Geometry of subRiemannian groups}'), and by the European Research Council (ERC Starting Grant 713998 GeoMeG `\emph{Geometry of Metric Groups}').}

\maketitle
\tableofcontents
\typeout{************************************************}
\typeout{Section 1 Introduction}
\typeout{************************************************}

\section{Introduction}\label{x:section:section-introduction}
Carnot groups have rich algebraic and metric structures, and share many properties with normed spaces. Recently several articles have generalized classical regularity results of isometric embeddings in normed spaces into the setting of Carnot groups. In real normed spaces, there are two simple criteria for an isometric embedding to be affine: surjectivity or strict convexity of the norm on the target. Both regularity criteria have analogues for isometric embeddings of Carnot groups.

Surjective isometric embeddings behave in the Carnot group case similarly as they do in the normed-space case. Namely, isometries between arbitrary (open subsets of) Carnot groups are affine \cite{Le_Donne-Ottazzi-2017-carnot_isometries}, i.e.\@, compositions of left translations and group homomorphisms. For globally defined isometries, there is an even more general result that isometries between connected nilpotent metric Lie groups are affine \cite{Kivioja-Le_Donne-2017-nilpotent_isometries}.

For non-surjective isometric embeddings, it was proved in \cite{Kishimoto-2003-carnot-geodesics-and-isometries} that if \(G\) is a sub-Riemannian Carnot group of step 2, then all isometric embeddings \(\RR\hookrightarrow G\), i.e.\@, all \terminology{infinite geodesics}, are affine. This property was coined the \terminology{geodesic linearity property} in \cite{Balogh-Fassler-Sobrino-2018-embeddings_into_heisenberg}, and was used as an alternative to the strict convexity criterion as the two conditions are equivalent in normed spaces. More precisely, it was shown in \cite{Balogh-Fassler-Sobrino-2018-embeddings_into_heisenberg} that if \(\HH^n\) is a Heisenberg group with a homogeneous distance satisfying the geodesic linearity property, then all isometric embeddings \(\RR^m\hookrightarrow \HH^n\) and \(\HH^m\hookrightarrow\HH^n\) are affine.

It was conjectured in \cite{Balogh-Fassler-Sobrino-2018-embeddings_into_heisenberg} and subsequently proved in \cite{Balogh-Calogero-2018-infinite_heisenberg_geodesics} that for Heisenberg groups the geodesic linearity property is equivalent to strict convexity of the projection norm, see {Definition~\ref{x:definition:def-projection-norm}}. While the point of view presented in \cite{Balogh-Fassler-Sobrino-2018-embeddings_into_heisenberg} is purely metric, the essential tools of the proof in \cite{Balogh-Calogero-2018-infinite_heisenberg_geodesics} arise from considering an isometric embedding \(\RR\hookrightarrow \HH^n\) as an optimal control problem, and reformulating the first order necessary criterion of the Pontryagin Maximum Principle in the language of convex analysis.

The goal of this paper is to extend the main results of \cite{Balogh-Fassler-Sobrino-2018-embeddings_into_heisenberg} and \cite{Balogh-Calogero-2018-infinite_heisenberg_geodesics} to arbitrary Carnot groups of step 2. The central object of study is again the Pontryagin Maximum Principle, from which relevant invariants will be extracted by an asymptotic study of the optimal controls. The main result of the paper is the following:
\begin{theorem}\label{x:theorem:thm-infinite-geodesics-are-lines}
In every {sub-Finsler Carnot group} of step 2 with a strictly convex norm, every infinite {geodesic} is affine.
\end{theorem}
\begin{corollary}\label{x:corollary:cor-homogeneous-group-geodesics-are-lines}
Let \(G\) be a {stratified group} of step 2 equipped with a {homogeneous distance} \(d\) such that the {projection norm} of \(d\) is strictly convex. Then every infinite {geodesic} in \((G,d)\) is affine.
\end{corollary}
The necessity of the strict convexity assumption is a direct consequence of the necessity of strict convexity for linearity of geodesics in the normed-space case, see {Proposition~\ref{x:proposition:prop-existence-of-non-line-infinite-geodesic}}. Examples may also be found from the singular geodesics for the non-strictly convex \(\ell^\infty\) sub-Finsler norm exhibited in \cite{BBLDS-2017-sub-finsler-control-viewpoint} and \cite{Ardentov-Le_Donne-Sachkov-2019-subfinsler_cartan}.

The restriction to step 2 is motivated by the known counterexample in the simplest Carnot group of step 3, the sub-Riemannian Engel group. The complete study of geodesics in the sub-Riemannian Engel group in \cite{Ardentov-Sachkov-2015-Engel_cut_time} gives the first (and to date essentially only) known example of a non-affine infinite geodesic in a sub-Riemannian Carnot group. Note that in general, very little is known about geodesics even in the sub-Riemannian case, see \cite{Hakavuori-Le_Donne-2016-corners,Monti-Pigati-Vittone-2018-existence_of_tangent_lines,Hakavuori-Le_Donne-2018-cones,BCJPS-2018-rank_2_abnormals,BFPR-2018-3d_Sard} for some recent results.

The proof for Heisenberg groups in \cite{Balogh-Fassler-Sobrino-2018-embeddings_into_heisenberg} that the geodesic linearity property of the target implies that all isometric embeddings are affine works also more generally for stratified groups, see {Proposition~\ref{x:proposition:prop-affine-geodesics-implies-affine-embeddings}}. Consequently, {Corollary~\ref{x:corollary:cor-homogeneous-group-geodesics-are-lines}} leads to the analogous rigidity result for arbitrary isometric embeddings:
\begin{theorem}\label{x:theorem:thm-isometric-embeddings}
Let \((H,d_H)\) and \((G,d_G)\) be {stratified groups} with {homogeneous distances} such that \(G\) has step 2 and the {projection norm} of \(d_G\) is strictly convex. Then every isometric embedding \((H,d_H)\hookrightarrow (G,d_G)\) is affine.
\end{theorem}
It is worth remarking that although there are no explicit restrictions on the domain \((H,d_H)\) in {Theorem~\ref{x:theorem:thm-isometric-embeddings}}, the mere existence of an isometric embedding \((H,d_H)\hookrightarrow (G,d_G)\) implies some restrictions. In particular, Pansu's Rademacher theorem \cite{Pansu-1989-metriques_et_quasiisometries} implies that there must exist an injective homogeneous homomorphism \(H\to G\). It follows that \(H\) has step at most 2 and rank at most the rank of \(G\).

\typeout{************************************************}
\typeout{Subsection 1.1 Structure of the paper}
\typeout{************************************************}

\subsection{Structure of the paper}\label{g:subsection:idm105}
{Section~\ref{x:section:sec-definitions}} presents the relevant definitions that will be used throughout the rest of the paper and some basic lemmas. The main points of interest are properties of blowdowns of geodesics, i.e.\@, geodesics ``viewed from afar'', and the observations about subdifferentials of convex functions.

{Sections~\ref{x:section:sec-pmp}\textendash{}\ref{x:section:sec-infinite-geodesics-proof-conclusion}} are devoted to the proofs of {Theorem~\ref{x:theorem:thm-infinite-geodesics-are-lines}} and {Corollary~\ref{x:corollary:cor-homogeneous-group-geodesics-are-lines}} about infinite geodesics. {Section~\ref{x:section:sec-pmp}} rephrases the classical first order optimality condition of the Pontryagin Maximum Principle in the setting of a step 2 sub-Finsler Carnot group. In the sub-Riemannian case the PMP reduces to a linear ODE for the controls. This is no longer true in the sub-Finsler case, making explicit solution of the system unfeasible. Nonetheless, the PMP has a form ({Proposition~\ref{x:proposition:prop-step-2-PMP}}) that is well suited to the study of the asymptotic behavior of optimal controls. The key object is the bilinear form \(B\colon V_1\times V_1\to\RR\).

{Section~\ref{x:section:sec-asymptotics}} covers the aforementioned asymptotic study. The goal of the section is to study blowdowns of infinite geodesics through the behavior of their controls. Using integral averages of controls, it is shown that any blowdown control must in fact be contained in the kernel of the bilinear form \(B\).

{Section~\ref{x:section:sec-infinite-geodesics-proof-conclusion}} wraps up the proof of {Theorem~\ref{x:theorem:thm-infinite-geodesics-are-lines}} using the conclusions of the previous sections. This section is where the strict convexity of the norm is critical. The importance of the assumption is that any linear map has a unique maximum on the ball. By observing that any element of \(\ker B\) defines an invariant along the corresponding optimal control, the uniqueness is exploited to prove that infinite geodesics must be invariant under blowdowns. {Corollary~\ref{x:corollary:cor-homogeneous-group-geodesics-are-lines}} follows from the sub-Finsler case by the observation that the length metric associated with a homogeneous norm is always a sub-Finsler metric.

{Section~\ref{x:section:sec-isometric-embeddings}} covers the proof of {Theorem~\ref{x:theorem:thm-isometric-embeddings}} about isometric embeddings as a consequence of {Corollary~\ref{x:corollary:cor-homogeneous-group-geodesics-are-lines}}. The link between geodesics and general isometric embeddings arises from considering a foliation by horizontal lines in the domain and studying the induced foliation by infinite geodesics in the image. The affinity of isometric embeddings follows from the observation that two lines are at a sublinear distance from each other if and only if they are parallel.

\typeout{************************************************}
\typeout{Section 2 Preliminaries}
\typeout{************************************************}

\section{Preliminaries}\label{x:section:sec-definitions}

\typeout{************************************************}
\typeout{Subsection 2.1 Stratified groups and homogeneous distances}
\typeout{************************************************}

\subsection{Stratified groups and homogeneous distances}\label{g:subsection:idm132}
\begin{definition}\label{x:definition:def-stratified-group}
A \terminology{stratified group} is a Lie group \(G\) whose Lie algebra has a decomposition \(\mathfrak{g}=V_1\oplus V_2\oplus\dots\oplus V_s\) such that \(V_s\neq \{0\}\) and \([V_1,V_k] = V_{k+1}\) for all \(k=1,\dots,s\), with the convention that \(V_{s+1}=\{0\}\). The \terminology{rank} and \terminology{step} of the stratified group \(G\) are the integers \(r=\dim V_1\) and \(s\) respectively.
\label{g:notation:idm149}
\label{g:notation:idm152}
\label{g:notation:idm155}
\end{definition}
\begin{definition}\label{x:definition:def-dilation}
A \terminology{dilation} by a factor \(\dilationfactor\in\RR\) on a {stratified group} \(G\) is the Lie group automorphism \(\delta_{\dilationfactor}\colon G\to G\) defined for any \(X=X_1+\dots+X_s\in V_1\oplus\dots\oplus V_s\) by
\begin{equation*}
\delta_{\dilationfactor}\exp(X_1+X_2+\dots+X_s) = \exp(\dilationfactor X_1 + \dilationfactor^2X_2+\dots+\dilationfactor^sX_s)\text{.}
\end{equation*}

\label{g:notation:idm168}
\end{definition}
\begin{definition}\label{x:definition:def-homogeneous-distance}
A \terminology{homogeneous distance} on a {stratified group} \(G\) is a left-invariant distance \(d\), which is one-homogeneous with respect to the {dilations}, i.e.\@, which satisfies
\begin{equation*}
d(\delta_{\dilationfactor}(g),\delta_{\dilationfactor}(h)) = \dilationfactor d(g,h)\quad\forall \dilationfactor>0,\, \forall g,h\in G\text{.}
\end{equation*}

\end{definition}

\typeout{************************************************}
\typeout{Subsection 2.2 The projection norm}
\typeout{************************************************}

\subsection{The projection norm}\label{g:subsection:idm182}
\begin{definition}\label{x:definition:def-projection-norm}
Let \(G\) be a {stratified group} and let \(d\) be a {homogeneous distance} on \(G\). The \terminology{projection norm} associated with the homogeneous distance \(d\) is the function
\begin{equation*}
\norm{\cdot}_d\colon V_1\to\RR,\quad \norm{X}_d = d(e,\exp(X))\text{,}
\end{equation*}
where \(e\) is the identity element of the group \(G\).
\label{g:notation:idm197}
\end{definition}
It is not immediate that \(\norm{\cdot}_d\) defines a norm. In the setting of the Heisenberg groups, this is proved in \cite[Proposition~2.8]{Balogh-Fassler-Sobrino-2018-embeddings_into_heisenberg}. Their proof works with minor modification for any homogeneous distances in arbitrary stratified groups and is captured in the following lemmas. The triangle inequality of \(\norm{\cdot}_d\) is the only non-trivial part. In order to make use of the triangle inequality of the distance \(d\), the following distance estimate is used.
\begin{lemma}\label{x:lemma:lemma-projection-is-submetry}
Let \(\pi_{V_1}\colon \mathfrak{g}=V_1\oplus\dots\oplus V_s\to V_1\) be the projection with respect to the direct sum decomposition. Then
\begin{equation*}
\norm{X}_d\leq d(e,\exp(X+Y))\quad\forall X\in V_1,\,\forall Y\in [\mathfrak{g},\mathfrak{g}]\text{,}
\end{equation*}
so the horizontal projection \(\pi = \pi_{V_1}\circ\log\colon (G,d)\to (V_1,\norm{\cdot}_d)\) is a submetry.
\end{lemma}
\begin{proof}\label{g:proof:idm211}
Observe first that for any \(X\in V_1\) and \(Y=Y_2+\dots+Y_s\in V_2\oplus\dots\oplus V_s=[\mathfrak{g},\mathfrak{g}]\), and any \(n\in\NN\), homogeneity and the triangle inequality imply that
\begin{align*}
nd(e,\exp(X+\frac{1}{n}Y_2 + \dots+\frac{1}{n^{s-1}}Y_s))
&= d(e,\exp(nX+nY))\\
&\leq nd(e,\exp(X+Y))\text{.}
\end{align*}
Continuity of the distance then gives the bound
\begin{align*}
d(e,\exp(X)) &= \lim\limits_{n\to\infty} d(e,\exp(X+\frac{1}{n}Y_2 + \dots+\frac{1}{n^{s-1}}Y_s))\\
&\leq d(e,\exp(X+Y))
\end{align*}
for any \(X\in V_1\) and \(Y\in [\mathfrak{g},\mathfrak{g}]\) as claimed.

The previous estimate implies the containment \(\pi(B(e,r)) \subset B_{\norm{\cdot}_d}(0,r)\) for the projection of any ball \(B(e,r)\subset G\). On the other hand, {Definition~\ref{x:definition:def-projection-norm}} of the projection norm directly implies the opposite containment
\begin{equation*}
B_{\norm{\cdot}_d}(0,r) = V_1\cap \log B(e,r) \subset \pi(B(e,r))\text{.}
\end{equation*}
By left-invariance of the distance \(d\) it follows that the map \(\pi\) is a submetry.
\end{proof}
\begin{remark}\label{g:remark:idm231}
The estimate of {Lemma~\ref{x:lemma:lemma-projection-is-submetry}} is in general false for non-homogeneous left-invariant distances. Examples of the failure may be found by taking any {homogeneous metric} and tilting the decomposition \(V_1\oplus [\mathfrak{g},\mathfrak{g}]\).

Namely, let \(d\) be any {homogeneous distance} on a {stratified group} \(G\) for which the inequality of {Lemma~\ref{x:lemma:lemma-projection-is-submetry}} is strict when \(Y\neq 0\), for instance a sub-Riemannian distance. Define a new stratification for the Lie group \(G\) by replacing the first layer \(V_1\) with \(\tilde{V}_1\), where some vector \(X\in V_1\) is replaced by \(X+Y\) for some central vector \(Y\in [\mathfrak{g},\mathfrak{g}]\). As the Lie brackets and group law are unchanged, \(d\) is a left-invariant distance for the resulting stratified group \(\tilde{G}\), but is no longer homogeneous due to the tilting of the layers. In this way, the notion of projection to the first layer is changed, and the estimate of {Lemma~\ref{x:lemma:lemma-projection-is-submetry}} fails in \((\tilde{G},d)\) for the vectors \(X+Y\in\tilde{V_1}\) and \(-Y\in[\mathfrak{g},\mathfrak{g}]\).
\end{remark}
\begin{lemma}\label{x:lemma:lemma-projection-norm-is-a-norm}
The {projection norm} is a norm.
\end{lemma}
\begin{proof}\label{g:proof:idm259}
Positivity and homogeneity of the {projection norm} \(\norm{\cdot}_d\) follow immediately from positivity and homogeneity of the {homogeneous distance} \(d\). For the triangle inequality, let \(X,X'\in V_1\) and let \(Y\in [\mathfrak{g},\mathfrak{g}]\) be the element given by the Baker-Campbell-Hausdorff formula such that
\begin{equation*}
\exp(X)\exp(X') = \exp(X+X'+Y)\text{.}
\end{equation*}
{Lemma~\ref{x:lemma:lemma-projection-is-submetry}} gives the bound \(\norm{X+X'}_d\leq d(e,\exp(X+X'+Y))\). By the choice of \(Y\), the left-invariance and triangle inequality of \(d\) conclude the claim:
\begin{equation*}
d(e,\exp(X+X'+Y))= d(e,\exp(X)\exp(X'))\leq \norm{X}_d+\norm{X'}_d\text{.}\qedhere
\end{equation*}

\end{proof}

\typeout{************************************************}
\typeout{Subsection 2.3 Length structures and sub-Finsler Carnot groups}
\typeout{************************************************}

\subsection{Length structures and sub-Finsler Carnot groups}\label{g:subsection:idm273}
\begin{definition}\label{x:definition:def-length-metric}
Let \((X,d)\) be a metric space. Let \(\Omega\) be the space of rectifiable curves of \(X\) and let \(\ell_d\colon\Omega\to \RR\) be the length functional. For points \(x,y\in X\), denote by \(\Omega(x,y)\subset\Omega\) the space of all rectifiable curves connecting the points \(x\) and \(y\). The \terminology{length metric associated with the metric \(d\)} is the map \(d_\ell\colon X\times X\to\RR\cup\{\infty\}\) defined by
\begin{equation*}
d_\ell(x,y) := \inf \{\ell_d(\gamma): \gamma\in \Omega(x,y)   \}\text{.}
\end{equation*}
If \(d=d_\ell\), then the metric \(d\) is called a \terminology{length metric}.
\end{definition}
See \cite[Section~2.3]{Burago-Burago-Ivanov-2001-metric_geometry} for further information about length structures induced by metrics. For the purposes of this paper, only the special case of the length metric associated with a homogeneous distance will be relevant. Such a length metric always determines a sub-Finsler Carnot group, see {Definition~\ref{x:definition:def-subfinsler}} and {Lemma~\ref{x:lemma:lemma-length-metric-of-homogeneous-distance-is-subfinsler}}.
\begin{definition}\label{x:definition:def-control}
Let \(G\) be a {stratified group}. Denote by \(L_g\colon G\to G\) the left-translation \(L_g(h)=gh\). An absolutely continuous curve \(\gamma\colon [0,T]\to G\) is a \terminology{horizontal curve} if \((L_{\gamma(t)^{-1}})_*\dot{\gamma}(t) \in V_1\) for almost every \(t\in [0,T]\). The \terminology{control} of a horizontal curve \(\gamma\) is its left-trivialized derivative, i.e.\@, the map
\begin{equation*}
u\colon [0,T]\to V_1,\quad u(t) = (L_{\gamma(t)^{-1}})_*\dot{\gamma}(t)\text{.}
\end{equation*}

\label{g:notation:idm312}
\label{g:notation:idm316}
\end{definition}
\begin{definition}\label{x:definition:def-subfinsler}
A \terminology{sub-Finsler Carnot group} is a {stratified group} \(G\) equipped with a norm \(\norm{\cdot}\colon V_1\to\RR\). The norm induces a {homogeneous distance} \(d_{SF}\) via the {length structure} induced by \(\norm{\cdot}\) over horizontal curves.

More explicitly, for a {horizontal curve} \(\gamma\colon[0,T]\to G\) with {control} \(u\colon [0,T]\to V_1\), define the length
\begin{equation*}
\ell_{\norm{\cdot}}(\gamma) = \int_{0}^{T}\norm{u(t)}\,dt\text{.}
\end{equation*}
For \(g,h\in G\), let \(\Omega(g,h)\) be the family of all horizontal curves connecting \(g\) and \(h\). The sub-Finsler distance \(d_{SF}\) is defined as
\begin{equation*}
d_{SF}(g,h) := \inf \{ \ell_{\norm{\cdot}}(\gamma): \gamma\in \Omega(g,h)   \}\text{.}
\end{equation*}

\end{definition}

\typeout{************************************************}
\typeout{Subsection 2.4 Geodesics and blowdowns}
\typeout{************************************************}

\subsection{Geodesics and blowdowns}\label{g:subsection:idm342}
\begin{definition}\label{x:definition:def-geodesic}
Let \(G\) be a {stratified group} equipped with a {homogeneous distance} \(d\). A \terminology{geodesic} is an isometric embedding \(\gamma\colon[0,T]\to (G,d)\). That is, a geodesic satisfies
\begin{equation*}
d(\gamma(t),\gamma(s)) = \abs{t-s}\quad\forall t,s\in[0,T]\text{.}
\end{equation*}

\end{definition}
In the proof of {Theorem~\ref{x:theorem:thm-isometric-embeddings}} it will be convenient to consider also curves which preserve distances up to a constant factor. A curve \(\gamma\colon[0,T]\to (G,d)\) for which there exists some constant \(C>0\) such that
\begin{equation*}
d(\gamma(t),\gamma(s)) = C\abs{t-s}\quad\forall t,s\in[0,T]
\end{equation*}
will be called a \terminology{geodesic with speed \(C\)}.
\begin{lemma}\label{x:lemma:lemma-rescaled-control}
Let \(\gamma\colon [0,\infty)\to G\) be a horizontal curve with {control} \(u\colon [0,\infty)\to V_1\) and let \(\dilationfactor>0\) be a dilating factor. Then the dilated and reparametrized curve
\begin{equation*}
\gamma_{\dilationfactor} \colon [0,\infty)\to G,\quad \gamma_{\dilationfactor}(t):=\delta_{1/\dilationfactor}\gamma(\dilationfactor t)\text{,}
\end{equation*}
has the control
\begin{equation*}
u_{\dilationfactor}\colon [0,\infty)\to V_1,\quad u_{\dilationfactor}(t) := u(\dilationfactor t)\text{.}
\end{equation*}

\label{g:notation:idm370}
\end{lemma}
\begin{proof}\label{g:proof:idm374}
Since the dilations are group homomorphisms, the claim follows directly by the chain rule and {Definition~\ref{x:definition:def-control}} of a control:
\begin{equation*}
\frac{d}{dt}\gamma_{\dilationfactor}(t)
= (\delta_{1/\dilationfactor})_*\frac{d}{dt}\gamma(\dilationfactor t)
= (\delta_{1/\dilationfactor})_*(L_{\gamma(\dilationfactor t)})_*u(\dilationfactor t)\dilationfactor
= (L_{\gamma_{\dilationfactor}(t)})_*u_{\dilationfactor}(t)\text{.}\qedhere
\end{equation*}

\end{proof}
\begin{definition}\label{x:definition:def-blowdown}
Let \(\gamma\colon[0,\infty)\to G\) be a horizontal curve. Suppose for some sequence of scales \(\dilationfactor_k\to\infty\) the pointwise limit
\begin{equation*}
\tilde{\gamma}\colon[0,\infty)\to G,\quad \tilde{\gamma}(t) = \lim\limits_{k\to\infty}\gamma_{\dilationfactor_k}(t) = \lim\limits_{k\to\infty}\delta_{1/\dilationfactor_k}\gamma(\dilationfactor_k t)
\end{equation*}
exists. Such a curve \(\tilde{\gamma}\) is called a \terminology{blowdown} of the curve \(\gamma\) along the sequence of scales \(\dilationfactor_k\).
\end{definition}
\begin{remark}\label{g:remark:idm388}
If the curve \(\gamma\) is \(L\)-Lipschitz, then the curves \(\gamma_\dilationfactor\) are also all \(L\)-Lipschitz. Hence by Arzelà-Ascoli, up to taking a subsequence a blowdown along a sequence of scales will always exist.\end{remark}
\begin{lemma}\label{x:lemma:lemma-blowdown-properties}
Let \(\gamma\colon [0,\infty)\to G\) be an infinite {geodesic} and let \(\tilde{\gamma} = \lim\limits_{k\to\infty}\gamma_{\dilationfactor_k}\) be any {blowdown} of the curve \(\gamma\). Let \(u\) and \(\tilde{u}\) be the {controls} of the curves \(\gamma\) and \(\tilde{\gamma}\) respectively. Then
\begin{enumerate}[label=(\roman*)]
\item\label{x:li:lemma-blowdown-geodesic}The curve \(\tilde{\gamma}\) is an infinite geodesic.
\item\label{x:li:lemma-blowdown-control-convergence}Up to taking a subsequence, the {dilated controls} \(u_{\dilationfactor_k}\) converge to the control \(\tilde{u}\) in \(\Lloc([0,\infty);V_1)\).
\end{enumerate}

\end{lemma}
\begin{proof}\label{g:proof:idm414}
\textbf{(i).}
The curve \(\tilde{\gamma}\) is a geodesic as the pointwise limit of geodesics.

\textbf{(ii).}
The claim follows from \cite[Remark~3.13]{Monti-Pigati-Vittone-2018-tangent_cones}. The point is that by weak compactness of closed balls in \(\Lloc([0,\infty);V_1)\) there exists a weakly convergent subsequence \(u_\dilationfactor \rightharpoonup v\) to some \(v\in\Lloc([0,\infty);V_1)\). The definitions of control and weak convergence imply that \(v\) is a control for \(\tilde{\gamma}\), so in particular \(\tilde{u}(t)=v(t)\) for almost every \(t\). Finally, the geodesic assumption implies that \(\norm{u(t)}\equiv 1 \equiv \norm{\tilde{u}(t)}\), so the weak convergence is upgraded to strong convergence \(u_{\dilationfactor}\to \tilde{u}\) in \(\Lloc([0,\infty);V_1)\).

\end{proof}
\begin{lemma}\label{x:lemma:lemma-existence-of-blowdown-line}
Let \(G\) be a {sub-Finsler Carnot group} with a strictly convex norm and let \(\gamma\colon [0,\infty)\to G\) be an infinite {geodesic}. Then there exists a sequence \(\dilationfactor_k\to\infty\) such that the {blowdown} \(\tilde{\gamma} = \lim\limits_{k\to\infty}\gamma_{\dilationfactor_k}\) is affine.
\end{lemma}
\begin{proof}\label{g:proof:idm445}
If the geodesic \(\gamma\) is itself affine, then the claim is immediate. Suppose then that \(\gamma\) is not affine, i.e.\@, not a left translation of a one parameter subgroup. In particular, the geodesic \(\gamma\) has non-constant {control}. Hence the horizontal projection \(\pi\circ\gamma\colon [0,\infty)\to G/[G,G]\) has non-constant derivative, and is also not affine. Since \(G/[G,G]\) is a normed space with a strictly convex norm, the projection curve \(\pi\circ\gamma\) cannot be a geodesic. Then \cite[Theorem~1.4]{Hakavuori-Le_Donne-2018-cones} states that there exists a Carnot subgroup \(H \lt G\) of lower rank such that all blowdowns of the geodesic \(\gamma\) are contained in \(H\).

Let the curve \(\beta\colon [0,\infty)\to H \lt G\) be any blowdown. By {Lemma~\ref{x:lemma:lemma-blowdown-properties}}\ref{x:li:lemma-blowdown-geodesic}, \(\beta\) is a geodesic. If \(\beta\) is also not affine, then iterating the above there exists a Carnot subgroup \(K \lt H \lt G\) of even lower rank such that all blowdowns of \(\beta\) are in \(K\). Blowdowns of the geodesic \(\beta\) are also blowdowns of the geodesic \(\gamma\) by a diagonal argument, so the claim follows by induction, since a Carnot subgroup of rank 1 is just a one parameter subgroup.
\end{proof}

\typeout{************************************************}
\typeout{Subsection 2.5 Subdifferentials}
\typeout{************************************************}

\subsection{Subdifferentials}\label{x:subsection:subsec-subdifferentials}
In this section, let \(V\) be some fixed finite dimensional vector space and let \(E\colon V\to\RR\) be a convex continuous function. In the application in {Section~\ref{x:section:sec-infinite-geodesics-proof-conclusion}}, the space \(V\) will be the horizontal layer \(V_1\subset\mathfrak{g}\), and the convex function of interest will be a squared norm \(\frac{1}{2}\norm{\cdot}^2\).
\begin{definition}\label{x:definition:def-subdifferential}
A linear function \(a\colon V\to\RR\) is a \terminology{subdifferential} of the function \(E\) at a point \(Y\in V\) if
\begin{equation*}
a(X-Y)\leq E(X)-E(Y)\quad\forall X\in V\text{.}
\end{equation*}
The collection of all subdifferentials \(a\) at a point \(Y\in V\) is denoted \(\partial E(Y)\subset V^*\).
\label{g:notation:idm490}
\end{definition}
The following lemmas condense the properties of convex functions and their subdifferentials that will be relevant for the article. They will be utilized in the proof of {Theorem~\ref{x:theorem:thm-infinite-geodesics-are-lines}} in {Section~\ref{x:section:sec-infinite-geodesics-proof-conclusion}}. The first lemma is the continuity and compactness of subdifferentials.
\begin{lemma}\label{x:lemma:lemma-continuity-of-subdifferentials}
Let \(Y_k\to Y\in V\) be a converging sequence and let \(a_k\in\partial E(Y_k)\). Then there exists a converging subsequence \(a_k\to a\in\partial E(Y)\).
\end{lemma}
\begin{proof}\label{g:proof:idm504}
\cite[Theorem~24.7]{Rockafellar-1970-convex_analysis} shows (among other things) that since the set of points \(\mathcal{S}=\{Y_i: i\in\NN\}\cup \{Y\}\) is compact, the family of {subdifferentials}
\begin{equation*}
\partial E(\mathcal{S}):= \bigcup_{X\in \mathcal{S}}\{a\in \partial E(X) \}
\end{equation*}
is also compact. Hence there exists a converging subsequence \(a_k\to a\in\partial E(\mathcal{S})\). The claim is concluded by \cite[Theorem~24.4]{Rockafellar-1970-convex_analysis}, which shows that the convergences \(Y_k\to Y\) and \(a_k\to a\) with \(a_k\in\partial E(Y_k)\) imply that \(a\in\partial E(Y)\).
\end{proof}
The second lemma is a simple estimate on a subdifferential of the squared norm.
\begin{lemma}\label{x:lemma:lemma-subdifferential-of-norm}
Let \(\norm{\cdot}\) be a norm on \(V\) and let \(a\colon V\to\RR\) be a {subdifferential} of the squared norm \(E=\frac{1}{2}\norm{\cdot}^2\) at a point \(Y\in V\). Then \(\abs{a(X)}\leq \norm{X}\norm{Y}\) for all \(X\in V\), and \(a(Y)=\norm{Y}^2\).
\end{lemma}
\begin{proof}\label{g:proof:idm529}
For any points \(X,Y\in V\) and any \(\epsilon>0\), the subdifferential condition \(a\in\partial E(Y)\) implies that
\begin{align*}
\epsilon a(X)
&= a(Y+\epsilon X - Y)
\leq E(Y+\epsilon X) - E(Y)\\
&\leq\frac{1}{2}\Big((\norm{Y}+\epsilon\norm{X})^2 - \norm{Y}^2\Big)
= \epsilon \norm{X}\norm{Y} + \frac{1}{2}\epsilon ^2\norm{X}^2\text{.}
\end{align*}
Letting \(\epsilon\to 0\) proves the bound \(a(X)\leq \norm{X}\norm{Y}\). Repeating the same consideration for \(-X\), gives the opposite bound \(-a(X)\leq \norm{X}\norm{Y}\).

For the equality \(a(Y)=\norm{Y}^2\), let \(\epsilon>0\), and observe that a similar computation as before shows that
\begin{align*}
-\epsilon a(Y) &=
a((1-\epsilon)Y-Y)
\leq E((1-\epsilon)Y)-E(Y)\\
&= \frac{1}{2}((1-\epsilon)^2-1)\norm{Y}^2
= (-\epsilon+\frac{1}{2}\epsilon^2)\norm{Y}^2\text{.}
\end{align*}
That is, \(a(Y)\geq (1 - \frac{1}{2}\epsilon)\norm{Y}^2\). The limit as \(\epsilon\to 0\) and the previous upper bound prove the claim.
\end{proof}

\typeout{************************************************}
\typeout{Section 3 Step 2 sub-Finsler Pontryagin Maximum Principle}
\typeout{************************************************}

\section{Step 2 sub-Finsler Pontryagin Maximum Principle}\label{x:section:sec-pmp}
In this section, the Pontryagin Maximum Principle will be rephrased in a convenient form for the purposes of {Theorem~\ref{x:theorem:thm-infinite-geodesics-are-lines}}. The precise statement to be proved is the following:
\begin{proposition}[Step 2 sub-Finsler PMP.]\label{x:proposition:prop-step-2-PMP}
Let \(G\) be a step 2 {sub-Finsler Carnot group} with an arbitrary norm \(\norm{\cdot}\colon V_1\to\RR\) and let \(0\leq T\leq \infty\). If \(u\colon [0,T]\to V_1\) is the {control} of a {geodesic}, then there exists an absolutely continuous curve \(a\colon [0,T]\to V_1^*\) and a skew-symmetric bilinear form \(B\colon V_1\times V_1\to\RR\) such that
\begin{enumerate}[label=(\roman*)]
\item\label{x:li:prop-step-2-PMP-ode}At almost every \(t\in[0,T]\), the curve \(a\) has the derivative
\begin{equation*}
\frac{d}{dt}a(t)Y = B(u(t),Y)\quad\forall Y\in V_1\text{.}
\end{equation*}

\item\label{x:li:prop-step-2-PMP-subdiff}At almost every \(t\in[0,T]\), the linear map \(a(t)\colon V_1\to\RR\) is a {subdifferential} of the squared norm \(\frac{1}{2}\norm{\cdot}^2\) at the point \(u(t)\in V_1\).
\end{enumerate}

\end{proposition}
\begin{remark}\label{g:remark:idm578}
Up to changing the optimal control \(u\) on a set of measure zero, the subdifferential condition \ref{x:li:prop-step-2-PMP-subdiff} may be taken to hold for all \(t\in[0,T]\).

Namely, if condition \ref{x:li:prop-step-2-PMP-subdiff} holds on a subset \(I\subset[0,T]\) of full measure, for any \(t\in [0,T]\setminus I\), pick any converging sequence \(I\ni t_k\to t\) such that the limit \(\lim_{k\to\infty}u(t_k)\) exists, and redefine \(u(t)=\lim_{k\to\infty}u(t_k)\). By the continuity of subdifferentials given by {Lemma~\ref{x:lemma:lemma-continuity-of-subdifferentials}}, it follows that \(a(t)\) is a subdifferential of the squared norm at the point \(u(t)\).
\end{remark}
\begin{remark}\label{g:remark:idm593}
In the sub-Riemannian case, the squared norm \(\frac{1}{2}\norm{\cdot}^2\) is differentiable at every point, and the unique subdifferential is the inner product \(a(t) = \left\lt u(t),\cdot\right\gt\). The derivative condition \ref{x:li:prop-step-2-PMP-ode} then gives the usual linear ODE of controls in the implicit form
\begin{equation*}
\left\lt \dot{u}(t),Y\right\gt = \frac{d}{dt}\left\lt u(t),Y\right\gt = B(u(t),Y)\quad\forall Y\in V_1\text{.}
\end{equation*}

\end{remark}

\typeout{************************************************}
\typeout{Subsection 3.1 General statement of the PMP}
\typeout{************************************************}

\subsection{General statement of the PMP}\label{g:subsection:idm599}
For the rest of {Section~\ref{x:section:sec-pmp}}, let \(G\) be a fixed {sub-Finsler Carnot group} of step 2 with an arbitrary norm \(\norm{\cdot}\colon V_1\to\RR\), and let \(u\colon[0,T]\to V_1\) be the {control} of a geodesic \(\gamma\colon [0,T]\to G\).

Consider first the finite time \(T\lt \infty\) case. By {Definition~\ref{x:definition:def-subfinsler}} of the sub-Finsler distance, the control \(u\) minimizes the length functional \(\int_{0}^{T}\norm{u(t)}\,dt\) among all controls defining curves with the same endpoints as \(\gamma\). Since a {geodesic} has by definition constant speed, it follows that \(u\) is also a minimizer of the energy functional \(\frac{1}{2}\int_{0}^{T}\norm{u(t)}^2\,dt\).

Define the left-trivialized Hamiltonian
\begin{equation}
h\colon V_1\times\RR\times\mathfrak{g}^*\to\RR,\quad h(u,\xi,\lambda) = \lambda(u) + \frac{1}{2}\xi\norm{u}^2\text{.}\label{x:men:eq-hamiltonian}
\end{equation}
By the Pontryagin Maximum Principle as presented in \cite[Theorem~12.10]{Agrachev-Sachkov-2004-control_theory_geometric_viewpoint}, the {control} \(u\colon[0,T]\to V_1\) can minimize the energy \(\frac{1}{2}\int_{0}^{T}\norm{u(t)}^2\,dt\) only if there is an everywhere non-zero absolutely continuous dual curve \(t\mapsto (\xi,\lambda(t))\in \RR\times T^*_{\gamma(t)}G\) such that
\begin{align}
\xi &\leq 0\label{x:mrow:PMP-normality}\\
\dot{\lambda} &= \vec{h}_{u(t),\xi}(\lambda)\quad\text{a.e. }t\in[0,T],\label{x:mrow:PMP-hamiltonian-flow}\\
h_{u(t),\xi}(\lambda(t)) &\geq h_{v,\xi}(\lambda(t))\quad\forall v\in V_1\quad\text{a.e. }t\in[0,T].\label{x:mrow:PMP-subdifferential}
\end{align}
Here \(h_{v,\xi}\) and \(\vec{h}_{v,\xi}\), for \(v\in V_1\), are the left-invariant Hamiltonian and the associated Hamiltonian vector field respectively.

More explicitly, \(h_{v,\xi}\colon T^*G\to\RR\) is the function defined from the left-trivialized Hamiltonian {(\ref{x:men:eq-hamiltonian})} in the natural way by
\begin{equation*}
h_{v,\xi}(\lambda)=h(v,\xi,L_{g}^*\lambda),\quad \forall \lambda\in T^*_gG\text{,}
\end{equation*}
and \(\vec{h}_{v,\xi}\) is the Hamiltonian vector field associated with the left-invariant Hamiltonian \(h_{v,\xi}\) by duality through the canonical symplectic form on the cotangent bundle, see \cite[Section~4]{Agrachev-Barilari-Boscain-2019-subriemannian_book} for more details within the context of the PMP in the sub-Riemannian setting.

Observe that if \((\xi,\lambda(t))\) is a dual curve satisfying the conditions {(\ref{x:mrow:PMP-normality})\textendash{}(\ref{x:mrow:PMP-subdifferential})} of the PMP, then also any scalar multiple \((C\xi,C\lambda(t))\) for any \(C\gt 0\) satisfies the conditions {(\ref{x:mrow:PMP-normality})\textendash{}(\ref{x:mrow:PMP-subdifferential})} of the PMP. This observation allows the infinite time case \(T=\infty\) to be handled as a limit of the finite time case. Namely, if \(u\colon [0,\infty)\to V_1\) is the {control} of a {geodesic}, then all its finite restrictions \(\restr{u}{[0,k]}\colon [0,k]\to V_1\) for \(k\in\NN\) are also controls of geodesics, so by the above they have corresponding dual curves \(t\mapsto (\xi_k,\lambda_k(t))\). By taking suitable rescalings of the \((\xi_k,\lambda_k)\), there exists a non-zero limit \((\xi_\infty,\lambda_\infty)\), which then satisfies the conditions {(\ref{x:mrow:PMP-normality})\textendash{}(\ref{x:mrow:PMP-subdifferential})} of the PMP on the entire interval \([0,\infty)\).

Condition {(\ref{x:mrow:PMP-normality})} is a binary condition \(\xi=0\) or \(\xi\neq 0\). The case \(\xi=0\) is the case of an abnormal {control} \(u\), and may be ignored in the step 2 setting, since the second order necessary criterion of the Goh condition (see e.g.\@ \cite[Section~20]{Agrachev-Sachkov-2004-control_theory_geometric_viewpoint}) implies that there are no strictly abnormal extremals in step 2. By rescaling \((\xi,\lambda)\) it therefore suffices to consider the normal case \(\xi=-1\).

\typeout{************************************************}
\typeout{Subsection 3.2 The PMP in left-trivialized coordinates}
\typeout{************************************************}

\subsection{The PMP in left-trivialized coordinates}\label{g:subsection:idm667}
Let \(X_1,\ldots,X_r\) be a basis of \(V_1\). Fix a basis \(X_{r+1},\ldots,X_n\) for \(V_2=[V_1,V_1]\) by choosing a maximal linearly independent subset of the Lie brackets \(\{[X_i,X_j]: 1\leq i\lt j\leq r\}\). By an abuse of notation, denote also by \(X_1,\ldots,X_n\), the corresponding left-invariant frame of \(TG\). Let \(\theta_1,\ldots,\theta_n\) be the dual left-invariant frame of \(T^*G\). Writing the curve \(\lambda(t)\) in left-trivialized coordinates as
\begin{equation*}
\lambda(t) = \sum_{i=1}^n\lambda_i(t)\theta_i(\gamma(t))\text{,}
\end{equation*}
the Hamiltonian ODE {(\ref{x:mrow:PMP-hamiltonian-flow})} in the normal case \(\xi=-1\) takes the simpler form
\begin{equation}
\dot{\lambda}_i(t) = \lambda(t)\Big(\Big[\sum_{j=1}^ru_j(t)X_j, X_i\Big](\gamma(t))\Big),\quad i=1,\dots,n\text{,}\label{x:men:eq-left-trivialized-hamiltonian-flow}
\end{equation}
see \cite[Section~18.3]{Agrachev-Sachkov-2004-control_theory_geometric_viewpoint} for the explicit computation.
\begin{proof}[Proof of Proposition~\ref{x:proposition:prop-step-2-PMP}.]\label{g:proof:idm685}
The curve \(a\colon[0,T]\to V_1^*\) will be given by restricting the linear map
\begin{equation}
a(t) := (L_{\gamma(t)})^*\lambda(t)\colon \mathfrak{g}\to\RR\label{x:men:eq-subdifferential-curve-defn}
\end{equation}
to \(V_1\). The skew-symmetric bilinear form \(B\colon V_1\times V_1\to\RR\) will be given by
\begin{equation}
B(X,Y) := a(t)[X,Y]\text{.}\label{x:men:eq-bilinear-form-defn}
\end{equation}

The curve \(a(t)\) of {(\ref{x:men:eq-subdifferential-curve-defn})} has in the left-invariant frame the same coefficients as the curve \(\lambda(t)\), i.e.\@, the coefficients of \(a(t) = \sum_{i=1}^na_i(t)\theta_i(e)\) are exactly \(a_i=\lambda_i\). Left-translating the Hamiltonian ODE {(\ref{x:men:eq-left-trivialized-hamiltonian-flow})} to the identity shows that for almost every \(t\in[0,T]\), the components of the curve have the derivatives
\begin{equation}
\dot{a}_i(t) = \frac{d}{dt}\lambda_i(t) = a(t)[u(t),X_i],\quad i=1,\ldots,n\text{.}\label{x:men:eq-coordinate-ode}
\end{equation}

By the step 2 assumption, \([u(t),X_i]=0\) for all the vertical components \(i=r+1,\dots,n\), so the vertical coefficients \(a_{r+1},\dots,a_n\) are all constant. Therefore \(a(t)[X,Y] = \sum_{i=r+1}^{n}a_i\theta_i([X,Y])\) is constant in \(t\). That is, the expression {(\ref{x:men:eq-bilinear-form-defn})} defines a unique bilinear form \(B\) independent from \(t\).

Writing the system {(\ref{x:men:eq-coordinate-ode})} in terms of the bilinear form \(B\), the remaining non-trivial equations are exactly
\begin{equation*}
\dot{a}_i(t) = a(t)[u(t),X_i]
= B(u(t),X_i),\quad i=1,\ldots,r\text{.}
\end{equation*}
The derivative condition {\ref{x:proposition:prop-step-2-PMP}}\ref{x:li:prop-step-2-PMP-ode} follows by linearity, as for an arbitrary vector \(Y=y_1X_1+\dots+y_rX_r\in V_1\), the above implies that
\begin{equation*}
\frac{d}{dt}a(t)Y = \frac{d}{dt}\sum_{i=1}^na_i(t)y_i
= \sum_{i=1}^nB(u(t),X_i)y_i
= B(u(t),Y)\text{.}
\end{equation*}

The subdifferential condition {\ref{x:proposition:prop-step-2-PMP}}\ref{x:li:prop-step-2-PMP-subdiff} for the linear functions \(a(t)\) follows from rephrasing the maximality condition {(\ref{x:mrow:PMP-subdifferential})}. Namely, expanding out the explicit expressions of the normal Hamiltonians \(h_{u(t),-1}\) and \(h_{v,-1}\) from {(\ref{x:men:eq-hamiltonian})} and reorganizing terms, the maximality condition {(\ref{x:mrow:PMP-subdifferential})} is equivalently stated as
\begin{equation*}
a(t)v - a(t)u(t) \leq \frac{1}{2}\norm{v}^2-\frac{1}{2}\norm{u(t)}^2
\quad\forall v\in V_1\quad\text{a.e. }t\in[0,T]\text{.}
\end{equation*}
This is exactly {Definition~\ref{x:definition:def-subdifferential}} stating that the linear function \(a(t)\) is a subdifferential of the squared norm \(\frac{1}{2}\norm{\cdot}^2\) at the point \(u(t)\in V_1\).
\end{proof}

\typeout{************************************************}
\typeout{Section 4 Asymptotic behavior of controls}
\typeout{************************************************}

\section{Asymptotic behavior of controls}\label{x:section:sec-asymptotics}
In this section, let \(u\colon [0,\infty)\to V_1\) be a fixed {control} satisfying the {PMP~\ref{x:proposition:prop-step-2-PMP}}. Let \(a\colon[0,\infty)\to V_1^*\) be the associated curve of {subdifferentials} and let \(B\colon V_1\times V_1\to\RR\) be the associated bilinear form.
\begin{lemma}\label{x:lemma:lemma-limit-average-in-kernel}
For every vector \(X\in V_1\),
\begin{equation*}
\lim\limits_{T\to\infty}B\left(\intavg_{0}^Tu(t)\,dt,X\right) = 0\text{.}
\end{equation*}

\end{lemma}
\begin{proof}\label{g:proof:idm749}
Fix an arbitrary vector \(X\in V_1\). Bilinearity of the map \(B\) implies that
\begin{equation}
B\left(\intavg_{0}^Tu(t)\,dt,X\right)
= \frac{1}{T}\int_{0}^TB\left(u(t),X\right)\,dt\text{.}\label{x:men:eq-B-integral-average}
\end{equation}
Since the curve \(a\) is absolutely continuous, the derivative condition {PMP~\ref{x:proposition:prop-step-2-PMP}}\ref{x:li:prop-step-2-PMP-ode} implies that
\begin{equation}
\int_{0}^{T}B(u(t),X)
= \int_{0}^{T}\frac{d}{dt}a(t)X
= a(T)X-a(0)X\text{.}\label{x:men:eq-B-integral-bound-in-T}
\end{equation}

By the subdifferential condition {PMP~\ref{x:proposition:prop-step-2-PMP}}\ref{x:li:prop-step-2-PMP-subdiff}, for almost every \(T\), the linear map \(a(T)\) is a {subdifferential} of the squared norm \(\frac{1}{2}\norm{\cdot}^2\) at the point \(u(T)\). Since \(\norm{u(T)}\equiv 1\) is constant, continuity of the curve \(a\) and {Lemma~\ref{x:lemma:lemma-subdifferential-of-norm}} imply the bound \(\abs{a(T)X} \leq \norm{X}\) for every \(T\in[0,\infty)\). The identities {(\ref{x:men:eq-B-integral-average})} and {(\ref{x:men:eq-B-integral-bound-in-T})} then imply the desired conclusion that
\begin{equation*}
\lim\limits_{T\to\infty}\abs{B\left(\intavg_{0}^Tu(t)\,dt,X\right)}
\leq \lim\limits_{T\to\infty}\frac{2}{T}\norm{X}
= 0\text{.}\qedhere
\end{equation*}

\end{proof}
\begin{lemma}\label{x:lemma:lemma-asymptotic-cone-control-in-kernel}
Let \(\dilationfactor_k\to \infty\) be a diverging sequence and let \(u_{\dilationfactor_k}(t)=u(\dilationfactor_kt)\) be the corresponding {dilated controls}. If \(u_{\dilationfactor_k}\to \tilde{u}\) in \(\Lloc([0,\infty);V_1)\), then \(\tilde{u}(t)\in\ker B\) for almost every \(t\in[0,\infty)\).
\end{lemma}
\begin{proof}\label{g:proof:idm784}
By the Lebesgue differentiation theorem it suffices to prove that \(\intavg_{a}^b\tilde{u}(t)\,dt\in \ker B\) for any \(0\leq a\lt b\lt \infty\).

Fix \(0\leq a\lt b\lt \infty\). By assumption \(u_{\dilationfactor_k}\to \tilde{u}\) in \(\mathrm{L}^2([a,b];V_1)\), so there exists some error term \(\epsilon\colon \NN\to V_1\) with \(\lim\limits_{k\to\infty}\epsilon(k)=0\) such that
\begin{equation}
\intavg_{a}^b\tilde{u}(t)\,dt
= \intavg_a^bu(\dilationfactor_kt)\,dt +\epsilon(k)
= \intavg_{a\dilationfactor_k}^{b\dilationfactor_k}u(t)\,dt+\epsilon(k)\text{.}\label{x:men:eq-limit-intavg}
\end{equation}
The right-hand integral average can further be expressed as a difference of integral averages as
\begin{equation}
\intavg_{a\dilationfactor_k}^{b\dilationfactor_k}u(t)\,dt
= \frac{b}{b-a}\cdot\intavg_{0}^{b\dilationfactor_k}u(t)\,dt-\frac{a}{b-a}\cdot\intavg_{0}^{a\dilationfactor_k}u(t)\,dt\text{.}\label{x:men:eq-difference-of-intavg}
\end{equation}

{Lemma~\ref{x:lemma:lemma-limit-average-in-kernel}} implies that for any \(X\in V_1\)
\begin{equation*}
\lim\limits_{k\to\infty}B\left(\intavg_{0}^{b\dilationfactor_k}u(t)\,dt,X\right) = \lim\limits_{k\to\infty}B\left(\intavg_{0}^{a\dilationfactor_k}u(t)\,dt,X\right) = 0\text{.}
\end{equation*}
Combining the identities {(\ref{x:men:eq-limit-intavg})} and {(\ref{x:men:eq-difference-of-intavg})} and using bilinearity of \(B\) then implies that \(B\left(\intavg_a^b\tilde{u}(t)\,dt,X\right)=0\). Since the vector \(X\in V_1\) was arbitrary, this proves the desired claim that \(\intavg_{a}^b\tilde{u}(t)\,dt\in\ker B\).
\end{proof}

\typeout{************************************************}
\typeout{Section 5 Affinity of infinite geodesics}
\typeout{************************************************}

\section{Affinity of infinite geodesics}\label{x:section:sec-infinite-geodesics-proof-conclusion}

\typeout{************************************************}
\typeout{Subsection 5.1 Sub-Finsler Carnot groups}
\typeout{************************************************}

\subsection{Sub-Finsler Carnot groups}\label{g:subsection:idm808}
The proof of {Theorem~\ref{x:theorem:thm-infinite-geodesics-are-lines}} will now be concluded. The key ingredients are the {sub-Finsler PMP~\ref{x:proposition:prop-step-2-PMP}}, the knowledge of asymptotic behavior of blowdown controls from {Lemma~\ref{x:lemma:lemma-asymptotic-cone-control-in-kernel}}, and the convex analysis arguments from {Subsection~\ref{x:subsection:subsec-subdifferentials}}.
\begin{proof}[Proof of Theorem~\ref{x:theorem:thm-infinite-geodesics-are-lines}.]\label{x:proof:proof-thm-subfinsler-infinite-geodesics}
Let \(\gamma\colon[0,\infty)\to G\) be an infinite {geodesic} and let \(u\colon[0,\infty)\to V_1\) be its {control}. Let \(a\colon[0,\infty)\to V_1^*\) be the curve of {subdifferentials} of the squared norm \(\frac{1}{2}\norm{\cdot}^2\) and let \(B\colon V_1\times V_1\to\RR\) be the skew-symmetric bilinear form given by the {PMP~\ref{x:proposition:prop-step-2-PMP}}.

By {Lemma~\ref{x:lemma:lemma-existence-of-blowdown-line}}, there exists a sequence \(\dilationfactor_k\to\infty\) such that the {blowdown} \(\tilde{\gamma}=\lim\limits_{k\to\infty}\delta_{1/\dilationfactor_k}\circ\gamma\circ\delta_{\dilationfactor_k}\colon [0,\infty)\to G\) is affine, i.e.\@, a left translation of a one parameter subgroup. By {Lemma~\ref{x:lemma:lemma-blowdown-properties}}, taking a subsequence if necessary, the {dilated controls} \(u_{\dilationfactor_k}(t) = u(\dilationfactor_kt)\) converge in \(\Lloc([0,\infty);V_1)\) to the control \(\tilde{u}\) of the curve \(\tilde{\gamma}\). Since the curve \(\tilde{\gamma}\) is affine, the control \(\tilde{u}\) is constant. That is, there exists a constant vector \(Y\in V_1\), which for almost every \(t\in[0,\infty)\) is the limit
\begin{equation}
Y = \tilde{u}(t) = \lim\limits_{k\to\infty}u(\dilationfactor_k t)\text{.}\label{x:men:eq-Y-as-limit}
\end{equation}
By {Lemma~\ref{x:lemma:lemma-asymptotic-cone-control-in-kernel}}, \(Y\in\ker B\), so the derivative condition {PMP~\ref{x:proposition:prop-step-2-PMP}}\ref{x:li:prop-step-2-PMP-ode} implies that the curve \(t\mapsto a(t)Y\) is constant \(a(t)Y\equiv: C\).

Fix any \(t\in[0,\infty)\) such that the limit {(\ref{x:men:eq-Y-as-limit})} holds. By {Lemma~\ref{x:lemma:lemma-continuity-of-subdifferentials}}, up to taking a further subsequence, the subdifferentials \(a(\dilationfactor_k t)\) of the squared norm \(\frac{1}{2}\norm{\cdot}^2\) at the points \(u(\dilationfactor_k t)\) converge to a subdifferential \(\tilde{a}\colon V_1\to\RR\) of the squared norm \(\frac{1}{2}\norm{\cdot}^2\) at the point \(Y\). Moreover, since \(a(t)Y\equiv C\) is constant, also the limit evaluates to \(\tilde{a}Y = C\). Applying {Lemma~\ref{x:lemma:lemma-subdifferential-of-norm}} for the subdifferential \(\tilde{a}\) shows that \(C = \tilde{a}Y = \norm{Y}^2\). Similarly applying {Lemma~\ref{x:lemma:lemma-subdifferential-of-norm}} for the subdifferential \(a(t)\) shows that \(a(t)u(t) = \norm{u(t)}^2\). Since the curves \(\gamma\) and \(\tilde{\gamma}\) are both geodesics, \(\norm{u(t)} = 1 = \norm{Y}\), so combining all of the above gives the equality
\begin{equation*}
a(t)Y = 1 = a(t)u(t)\text{.}
\end{equation*}
Consequently for any convex combination \(X\in V_1\) of \(u(t)\) and \(Y\), {Lemma~\ref{x:lemma:lemma-subdifferential-of-norm}} implies that
\begin{equation*}
1 = a(t)X \leq \norm{X}\norm{u(t)} = \norm{X}\text{.}
\end{equation*}
By strict convexity of the norm this is only possible when \(u(t)=Y\).

Repeating the same argument at all the times \(t\) satisfying the limit {(\ref{x:men:eq-Y-as-limit})}, it follows that \(u(t)=Y\) for almost every \(t\in[0,\infty)\), so the geodesic \(\gamma\) is affine.
\end{proof}

\typeout{************************************************}
\typeout{Subsection 5.2 Arbitrary homogeneous distances}
\typeout{************************************************}

\subsection{Arbitrary homogeneous distances}\label{g:subsection:idm885}
The proof of {Corollary~\ref{x:corollary:cor-homogeneous-group-geodesics-are-lines}} about infinite {geodesics} for arbitrary {homogeneous distances} follows from the {sub-Finsler case} by passing to the induced {length metric}. The relevant properties are captured in the next lemma.
\begin{lemma}\label{x:lemma:lemma-length-metric-of-homogeneous-distance-is-subfinsler}
Let \((G,d)\) be a {stratified group} equipped with a {homogeneous distance} \(d\) and let \(d_\ell\) be the {length metric} of \(d\). Then
\begin{enumerate}[label=(\roman*)]
\item\label{x:li:lemma-length-subfinsler}\((G,d_\ell)\) is a {sub-Finsler Carnot group}.
\item\label{x:li:lemma-length-geodesics}All {geodesics} of \((G,d)\) are also geodesics of \((G,d_\ell)\).
\item\label{x:li:lemma-length-subfinsler-norm}The {projection norm} of \(d\) is the {sub-Finsler norm} of \(d_\ell\).
\end{enumerate}

\end{lemma}
\begin{proof}\label{g:proof:idm916}
\textbf{(i).}
In \cite[Theorem~1.1]{Le_Donne-2015-metric_characterization} sub-Finsler Carnot groups are characterized as the only {geodesic} metric spaces that are locally compact, isometrically homogeneous, and admit a dilation. Therefore it suffices to verify that the {length metric} associated with a {homogeneous distance} satisfies these properties.

The claims of isometric homogeneity and admitting a dilation follow directly from the corresponding properties of the {metric} \(d\). Namely, since {left-translations} are isometries of the metric \(d\), they preserve the length of curves, and hence are also isometries of the {length metric} \(d_\ell\). Similarly since {dilations} scale the length of curves linearly, they are dilations for the length metric \(d_\ell\).

Finiteness of the {length metric} \(d_\ell\) follows from the {stratification} assumption: each element \(g\in G\) can be written as a product of elements in \(\exp(V_1)\) and the {horizontal} lines \(t\mapsto \exp(tX)\) are all {geodesics}. Therefore concatenation of suitable horizontal line segments defines a finite length curve from the identity \(e\) to any desired point \(g\). It follows that the length metric \(d_\ell\) determines a well defined {homogeneous distance} on \(G\), so by  \cite[Proposition~2.26]{Le_Donne-Rigot-2019-besicovitch-covering_on_graded_groups} it induces the manifold topology of \(G\). In particular \((G,d_\ell)\) is a boundedly compact length space, so it is a geodesic metric space (see  \cite[Corollary~2.5.20]{Burago-Burago-Ivanov-2001-metric_geometry}). Applying  \cite[Theorem~1.1]{Le_Donne-2015-metric_characterization} shows that \((G,d_\ell)\) is a {sub-Finsler Carnot group}.

\textbf{(ii).}
The lengths of all rectifiable curves in the original metric \(d\) and its associated {length metric} \(d_\ell\) always agree (see  \cite[Proposition~2.3.12]{Burago-Burago-Ivanov-2001-metric_geometry}). In particular, the claim that the {geodesics} of \((G,d)\) are geodesics of \((G,d_\ell)\) follows.

\textbf{(iii).}
The horizontal projection \(\pi\colon (G,d)\to V_1\) is a submetry both for the {sub-Finsler} norm \(\norm{\cdot}_{SF}\) (by definition) and for the {projection norm} \(\norm{\cdot}_d\) (by {Lemma~\ref{x:lemma:lemma-projection-is-submetry}}). Hence the norms \(\norm{\cdot}_{SF}\) and \(\norm{\cdot}_d\) have exactly the same balls, so \(\norm{\cdot}_{SF} = \norm{\cdot}_d\).

\end{proof}
\begin{proof}[Proof of Corollary~\ref{x:corollary:cor-homogeneous-group-geodesics-are-lines}.]\label{x:proof:proof-cor-homogeneous-group-geodesics}
Let \((G,d)\) be a {stratified group} of step 2 equipped with a {homogeneous distance} \(d\) whose {projection norm} is strictly convex, and let \(\gamma\colon[0,\infty)\to (G,d)\) be an infinite {geodesic}.

Let \(d_\ell\) be the {length-metric} associated with \(d\). By {Lemma~\ref{x:lemma:lemma-length-metric-of-homogeneous-distance-is-subfinsler}}\ref{x:li:lemma-length-subfinsler} and \ref{x:li:lemma-length-geodesics}, the curve \(\gamma\) is also a geodesic of \((G,\norm{\cdot})\), where \(\norm{\cdot}\colon V_1\to\RR\) is the sub-Finsler norm of the sub-Finsler metric \(d_\ell\). Moreover by {Lemma~\ref{x:lemma:lemma-length-metric-of-homogeneous-distance-is-subfinsler}}\ref{x:li:lemma-length-subfinsler-norm} the norm \(\norm{\cdot}=\norm{\cdot}_d\) is by assumption strictly convex.

Consequently by {Theorem~\ref{x:theorem:thm-infinite-geodesics-are-lines}}, the geodesic \(\gamma\) is affine.
\end{proof}
The necessity of the strict convexity assumption is an immediate consequence of the classical case of normed spaces by the following simple lifting argument.
\begin{proposition}\label{x:proposition:prop-existence-of-non-line-infinite-geodesic}
Let \(G\) be a {stratified group} equipped with an arbitrary {homogeneous distance} \(d\). If the {projection norm} of \(d\) is not strictly convex, then there exist an infinite {geodesic} \(\gamma\colon\RR\to G\) which is not affine.
\end{proposition}
\begin{proof}\label{g:proof:idm1019}
If the {projection norm} \(\norm{\cdot}_d\colon V_1\to\RR\) is not strictly convex, then there exists a non-linear {geodesic} \(\beta\colon\RR \to V_1\). For example, if the norm \(\norm{X+cY}_d\) is constant for \(-\epsilon\leq c\leq \epsilon\), then the curve \(\beta(t) = tX+\epsilon\sin(t)Y\) is an infinite geodesic.

By {Lemma~\ref{x:lemma:lemma-projection-is-submetry}}, the projection \(\pi\colon (G,d)\to (V_1,\norm{\cdot})\) is a submetry, so the geodesic \(\beta\colon \RR\to V_1\) lifts to an infinite geodesic \(\gamma\colon \RR\to G\). Since the projection is a homomorphism and the geodesic \(\beta\) is not affine, neither is the geodesic \(\gamma\).
\end{proof}

\typeout{************************************************}
\typeout{Section 6 Affinity of isometric embeddings}
\typeout{************************************************}

\section{Affinity of isometric embeddings}\label{x:section:sec-isometric-embeddings}
{Theorem~\ref{x:theorem:thm-isometric-embeddings}} about isometric embeddings being affine follows from {Corollary~\ref{x:corollary:cor-homogeneous-group-geodesics-are-lines}} by an abstraction of the argument of \cite[Theorem~4.1]{Balogh-Fassler-Sobrino-2018-embeddings_into_heisenberg}. The abstract version of their statement is {Proposition~\ref{x:proposition:prop-affine-geodesics-implies-affine-embeddings}}. The key link between the metric and algebraic properties is the following simple lemma stating that the distance between two lines grows sublinearly if and only if the lines are parallel.
\begin{lemma}\label{x:lemma:lemma-sublinear-geodesics}
Let \((G,d)\) be a {stratified group} with a {homogeneous distance}. Then for all points \(g,h\in G\) and all vectors \(X,Y\in V_1\)
\begin{equation*}
d(g\exp(tX),h\exp(tY)) = o(t)\text{ as }t\to \infty \iff X=Y\text{.}
\end{equation*}

\end{lemma}
\begin{proof}\label{g:proof:idm1051}
Consider {dilations} by \(1/t\). Since dilations are homomorphisms, continuity of the distance gives the limit
\begin{align*}
\lim\limits_{t\to\infty}\frac{d(g\exp(tX),h\exp(tY))}{t}
&= \lim\limits_{t\to\infty}d(\delta_{1/t}(g)\exp(X),\delta_{1/t}(h)\exp(Y))\\
&=d(\exp(X),\exp(Y))\text{.}\qedhere
\end{align*}

\end{proof}
\begin{proposition}\label{x:proposition:prop-affine-geodesics-implies-affine-embeddings}
Let \((H,d_H)\) and \((G,d_G)\) be {stratified groups} with {homogeneous distances} such that all infinite geodesics in \(G\) are affine. Then every isometric embedding \((H,d_H)\hookrightarrow (G,d_G)\) is affine.
\end{proposition}
\begin{proof}\label{g:proof:idm1067}
Let \(\varphi\colon (H,d_H)\hookrightarrow (G,d_G)\) be an isometric embedding. Since left-translations are isometries, it suffices to consider the case when the map \(\varphi\) preserves the identity element, and prove that such an isometric embedding is a homomorphism.

Consider an arbitrary point \(h\in H\) and a horizontal vector \(X\in V_1^H\). The horizontal line \(t\mapsto h\exp(tX)\) is an infinite geodesic with speed \(\norm{X}_H\) through the point \(h\in H\). The image of the line under the isometric embedding \(\varphi\) is an infinite geodesic in the group \(G\) through the point \(\varphi(h)\) with exactly the same speed. By assumption all infinite geodesics in the group \(G\) are horizontal lines, so there exists some vector \(Y\in V_1^G\) (a priori depending on the point \(h\) and the vector \(X\)) with \(\norm{X}_H=\norm{Y}_G\) such that
\begin{equation*}
\varphi(h\exp(tX)) = \varphi(h)\exp(tY)\quad\forall t\in\RR\text{.}
\end{equation*}

Consider then the two parallel infinite geodesics \(t\mapsto \exp(tX)\) and \(t\mapsto h\exp(tX)\) with speed \(\norm{X}_H\). Repeating the previous consideration, since the map \(\varphi\) was assumed to preserve the identity, there exists another horizontal direction \(Z\in V_1^G\) such that \(\varphi(\exp(tX))= \exp(tZ)\). By {Lemma~\ref{x:lemma:lemma-sublinear-geodesics}}, the distance between the two lines in the group \(H\) grows sublinearly. Since the map \(\varphi\) is an isometric embedding, also the distance between the image lines in the group \(G\) grows sublinearly. Hence applying {Lemma~\ref{x:lemma:lemma-sublinear-geodesics}} in the converse direction implies that \(Y=Z\). That is, the vector \(Y\in V_1^G\) does not depend on the point \(h\in H\), only on the vector \(X\in V_1^H\).

The above shows that there is a well defined map \(\varphi_*\colon V_1^H\to V_1^G\) such that \(\varphi(h\exp(X)) = \varphi(h)\exp(\varphi_* X)\). In particular,
\begin{equation}
\varphi(h_1h_2) = \varphi(h_1)\varphi(h_2)\quad \forall h_1\in H\;\forall h_2\in\exp(V_1^H)\text{.}\label{x:men:eq-partial-homomorphism}
\end{equation}
Since the group \(H\) is {stratified}, the subset \(\exp(V_1^H)\) generates the entire group \(H\). That is, any element \(h\in H\) can be written as a finite product of elements in \(\exp(V_1^H)\). Applying the identity {(\ref{x:men:eq-partial-homomorphism})} repeatedly using such decompositions shows that the map \(\varphi\) is a homomorphism.
\end{proof}
{Theorem~\ref{x:theorem:thm-isometric-embeddings}} follows directly by combining the statements of {Corollary~\ref{x:corollary:cor-homogeneous-group-geodesics-are-lines}} and {Proposition~\ref{x:proposition:prop-affine-geodesics-implies-affine-embeddings}}.

\subsection*{Acknowledgements}
The author wishes to thank Enrico Le Donne and Yuri Sachkov for helpful discussions on the control viewpoint to infinite geodesics that paved the way to the conclusion of the main proof. The author also wishes to thank Ville Kivioja for his help during the outset of the project in the study of the known results and their abstractions.

\bibliographystyle{amsalpha}
\bibliography{ref}

\end{document}